\documentclass[11pt]{amsart}
\usepackage{amsthm,amsfonts,amsmath,amscd,amssymb,epsfig}
\usepackage[arrow, matrix, curve]{xy}
\theoremstyle{plain}
\newtheorem{theorem}{Theorem}[section]

\newtheorem{corollary}[theorem]{Corollary}

\newtheorem{proposition}[theorem]{Proposition}
\newtheorem{lemma}[theorem]{Lemma}
\newtheorem{definition}[theorem]{Definition}
\theoremstyle{remark}
\newtheorem{remark}{Remark}

\newcommand{\R}{{\mathbb{R}}}

\newcommand{\Z}{{\mathbb{Z}}}

\newcommand{\ind}{{\rm {ind}}}
\newcommand{\Hom}{{\rm {Hom}}}
\newcommand{\W}{\mathcal W}

\title{The Topology of Spaces of Polygons}
\author{Michael Farber and Viktor Fromm}                 

\address{Department of Mathematical Sciences, University of Durham, UK}

\email{MichaelSFarber@googlemail.com}
\email{viktor.fromm@durham.ac.uk}

\date{April 11, 2011}
\begin{document}

\maketitle

\begin{abstract}
Let $E_{d}(\ell)$ denote the space of all closed $n$-gons in $\R^{d}$ (where $d\ge 2$) with sides of length $\ell_1, \dots, \ell_n$, viewed up to translations. 
The spaces $E_d(\ell)$ 
are parameterized by their length vectors $\ell=(\ell_1, \dots, \ell_n)\in \R^n_{>}$ encoding the length parameters. 
Generically, $E_{d}(\ell)$ is a closed smooth manifold of dimension $(n-1)(d-1)-1$ supporting an obvious action of the orthogonal group ${ {O}}(d)$. However, the
quotient space $E_{d}(\ell)/{{O}}(d)$ (the moduli space of shapes of $n$-gons) has singularities for a generic $\ell$, assuming that $d>3$; this quotient
is well understood in the low dimensional cases $d=2$ and $d=3$. 
Our main result in this paper states that for fixed $d\ge 3$ and $n\ge 3$, the diffeomorphism types of the manifolds $E_{d}(\ell)$ for varying generic vectors $\ell$ are in one-to-one correspondence with some combinatorial objects -- connected components of the complement of a finite collection of hyperplanes. This result is in the spirit of  a conjecture of K. Walker who raised a similar problem in the planar case $d=2$.
\end{abstract}

\section{Introduction} It is well-known that essentially any closed smooth manifold can be realized as the configuration space of a mechanism \cite{JS}, \cite{KM3}. 
One may therefore consider the following inverse problem:
{\it Is it possible to determine metric parameters of a mechanism knowing the topology of its configuration space?}
A statement in this spirit (see \cite{Wa}), 
was proven in \cite{Fa4} for configuration spaces of linkages in $\R^3$, while the case of planar linkages was settled in \cite{Sch}. 
These results answered positively the conjecture raised by Kevin Walker \cite{Wa} in 1985. 
In paper \cite{FHS} the inverse problem was solved for a class of spaces of polygonal chains. 

In this paper we consider the inverse problem for the spaces of polygons of arbitrary dimension $d\ge 2$. Let $\ell=(l_{1}, \dots, l_{n})\in \R^n_>$ be a vector (called {\it the length vector}) with positive real coordinates $l_{1}, \dots, l_{n}$. 
For every $d \geq 2$, consider the space $E_{d}(\ell)$ given by 
$$E_{d}(\ell)=\{(u_{1}, \dots, u_{n}) \in (S^{d-1})^{n}: \sum \limits_{k=1}^n l_{k}u_{k}=0\}.$$ 
The points of $E_d(\ell)$ can be understood as closed $n$-gons in $\R^{d}$ with sides of length $l_{1}, \dots, l_{n}$, viewed up to Euclidean translations.
The question we study in this paper is whether one may determine $\ell$ (up to certain natural equivalence) knowing the homeomorphism type of the manifold $E_d(\ell)$. 

The orthogonal groups ${\rm O}(d)$ and ${\rm {SO}}(d)$ act naturally on $E_d(\ell)$ and the quotient $E_d(\ell)/{\rm {SO}}(d)$ is the moduli space of shapes of closed 
$n$-gons in $\R^d$ with sides $l_1, \dots, l_n$. In the cases $d=2$ and $d=3$ this quotient is generically a closed smooth manifold which is well-understood \cite{HK}, \cite{KM1}, \cite{Ka}, \cite{Fa5}. 
However for $d>3$ the quotient space $E_d(\ell)/{\rm {SO}}(d)$ 
has singularities for a generic $\ell$. This explains our choice in this paper to study rather the manifold $E_d(\ell)$ and not the quotient $E_d(\ell)/{\rm {SO}}(d)$, since we intend to include the case $d>3$.

\begin{definition} 
A length vector $\ell$ is called {\it generic} if there is no subset $J \subset \{1, \dots, n\}$ so that $ \sum \limits_{j \in J}l_{j} = \sum \limits_{j \notin J}l_{j}.$
\end{definition}

For a generic length vector $\ell\in \R^n_>$ the space $E_{d}(\ell)$ is a closed smooth manifold of dimension 
$$\dim E_d(\ell) = (n-1)(d-1)-1,$$
see Proposition \ref{gen}.

Given a permutation $\sigma: \{1, \dots, n\}  \to \{1, \dots, n\}$ and a length vector $\ell\in\R^n_>$ one may define $\sigma(\ell)\in \R^n_>$
by $\sigma(\ell) =(l_{\sigma(1)}, \dots, l_{\sigma(n)})$. Obviously the map  $(u_1, \dots, u_n) \mapsto (u_{\sigma(1)}, \dots, u_{\sigma(n)})$ is a diffeomorphism 
$E_d(\ell) \to E_d(\sigma(\ell))$. Thus we see that the length vector $\sigma(\ell)$ obtained by permuting the coordinates of $\ell$ defines a diffeomorphic manifold. 

\begin{definition}
We say that two generic length vectors $\ell, \ell'\in \R^n_>$, where 
$\ell=(l_1, \dots, l_n)$ and $\ell'=(l'_1, \dots, l'_n)$, lie in the same chamber 
if for any subset $J \subset \{1, \dots, n\}$ 
one has $ \sum \limits_{j \in J}l_{j} > \sum \limits_{j \notin J}l_{j}$
if and only if
$ \sum \limits_{j \in J}l'_{j} >\sum \limits_{j \notin J}l'_{j}.$
\end{definition}

It is clear that the set of all vectors $\ell'$ lying in the same chamber with a given vector $\ell$ is a convex set. 

Our main result in this paper states:

\begin{theorem} \label{Wal}
Let $\ell,\ell'\in \R^n_>$ be two generic length vectors and let $d \geq 3$. The following conditions are equivalent:%\marginpar{what about $d=2$?}

(a) The manifolds $E_{d}(\ell)$ and $E_{d}(\ell')$ are ${\rm O}(d)$-equivariantly diffeomorphic;

(b) The cohomology rings $H^{*}(E_{d}(\ell);\Z_{2})$ and $H^{*}(E_{d}(\ell');\Z_{2})$ are isomorphic as graded rings.  

(c) The rings $H^{(d-1)*}(E_{d}(\ell);\Z_{2})$ and $H^{(d-1)*}(E_{d}(\ell');\Z_{2})$ are isomorphic.  

(d) For some permutation $\sigma: \{1, \dots, n\} \to \{1, \dots, n\}$, the length vectors $\ell$ and $\sigma(\ell')$ lie in the same chamber; 

\end{theorem}

The equivalence (a) $\Leftrightarrow$ (d) can be interpreted as follows. For any subset $J\subset \{1, \dots, n\}$ consider the hyperplane $H_J$ in $\R^n$ with coordinates 
$l_1, \dots, l_n$ given by the equation $$\sum_{i\in J} l_i = \sum_{i\notin J} l_i.$$ 
Varying $J$ we obtain $2^{n-1}$ hyperplanes $H_J$ (since $J$ and its complement determine the same hyperplane). The complement of the union of these hyperplanes
\begin{eqnarray}\label{compl} 
\R^n_> - \bigcup_J H_J\end{eqnarray}
consists of all generic length vectors and the permutation group of $n$ symbols $\Sigma_n$ acts naturally on it. 
Theorem \ref{Wal} states that for given $n$ and $d$ the ${\rm O}(d)$-equivariant diffeomorphism types of manifolds $E_d(\ell)$ for varying generic length vectors $\ell$, are in one-to-one correspondence with 
the $\Sigma_n$-orbits of the connected components of (\ref{compl}). Thus we obtain a complete classification of the spaces of polygons $E_d(\ell)$ in purely combinatorial terms. 

The projection $p: E_d(\ell)\to S^{d-1}$ given by $p(u_1, \dots, u_n)= u_n$ is a fibration and its fiber is the chain space ${\mathcal C}_d^n(\ell)$ studied in \cite{FHS}. 
The main result of \cite{FHS} requires that the length vector $\ell$ is {\it dominated}, see \cite{FHS}, page 2. This assumption is not present in Theorem \ref{Wal} dealing with the spaces of polygons instead of the chain spaces.

In the case $d=2$ the space of polygons $E_d(\ell)$ is diffeomorphic to the product $M_\ell\times S^1$ where $M_\ell=E_2(\ell)/{\rm SO}(2)$ is the moduli space of planar polygons. 
From \cite{FHS} and \cite{Sch} we know that two generic length vectors $\ell, \ell'\in \R^n_>$ lie in the same chamber iff the integral cohomology rings $H^\ast(M_\ell;\Z)$ and 
$H^\ast(M_{\ell'};\Z)$ are isomorphic. 

The proof of Theorem \ref{Wal} is given in \S 5. 
We invoke techniques quite similar to those used in \cite{Fa4} and in \cite{FHS}, applying Morse theory and exploiting the result of Gubeladze \cite{Gub} on isomorphism problem for monoidal rings. However in this paper we require one additional tool - a lacunary principle for Morse - Bott functions which is developed in section \S \ref{sec:lac}; we believe that it is of independent interest and can be used in many other situations. 

%The method of proof also allows us to compute the $\Z_{2}$-Betti numbers of the spaces $E_{d}(l)$. Using this computation, 

In section 6 below we give an example of two length vectors $\ell,\ell'$ so that the corresponding spaces $E_{d}(\ell)$ and $E_{d}(\ell')$ have identical $\Z_{2}$-Betti numbers 
but lie in different orbits of chambers under the permutation action. Thus,  one is unable to distinguish between different orbits of chambers using the Betti numbers alone, i.e. 
without exploiting the product structure of the cohomology ring $H^{*}(E_{d}(\ell);\Z_{2})$.

\section{A Lacunary Principle for Morse-Bott Functions}\label{sec:lac}

The classical lacunary principle in its simplest form states that a Morse function is perfect if  the Morse indices of all critical points are divisible by an integer $k\ge 2$. Recall that perfectness of a Morse function means that 
the Morse inequalities are satisfied as equalities. In this section we propose a generalization of the lacunary principle for Morse - Bott functions, which will be used in the proof of Theorem \ref{Wal}. 

\begin{definition} \label{lacunary}Let $k\ge 2$ be an integer. A closed manifold $C$ is called $k$-lacunary if the homology groups $H_{j}(C;\Z)$ have no torsion and are trivial in all dimensions $j$ which are not divisible by $k$. 
\end{definition}

\begin{proposition}\label{Lac} Let $M$ be a smooth compact manifold, possibly with boundary. 
Let $f: M\to \R$ be a smooth function which is nondegenerate in the sense of Bott. If $\partial M\not= \emptyset$ we will additionally assume that $\partial M$ coincides with the set of points where $f$ achieves its maximum and $df\not=0$ on $\partial M$. Suppose that 
for some $k\ge 2$, each connected critical submanifold 
$C\subset M$ of $f$ is $k$-lacunary (see Definition \ref{lacunary}) and the Morse - Bott index  $\ind_{f}(C)$ of $C$ is divisible by $k$. Then $f$ is perfect, i.e. 
\begin{eqnarray}\label{iso1}
H_{*}(M;\Z) \simeq \bigoplus_{C \subset {\rm {Crit}}(f)}H_{*-\ind_{f}(C)}(C;\Z) ,
\end{eqnarray}
where $C$ runs over the connected components of the set of critical points of $f$. 
\end{proposition}

\begin{proof} We will use induction on the number of critical levels of $f$. 
Let $t_0< t_1 < \dots < t_k$ be regular values of $f$ such that $f(M)\subset (t_0,t_k)$ and 
for $i=0, \dots, k-1$
each interval $(t_i, t_{i+1})$ contains a single critical value of $f$. Let $M_i$ denote $f^{-1}(-\infty, t_i]$. 
We want to show inductively that 
\begin{eqnarray}\label{iso}
H_{*}(M_i;\Z) \simeq \bigoplus_{f(C)<t_i}H_{*-\ind_{f}(C)}(C;\Z).
\end{eqnarray}

For $i=1$, the interval $(-\infty,t_1)$ contains exactly one critical value of $f$ (the minimum). Hence the desired isomorphism (\ref{iso}) follows from the fact that 
$M_1$ deformation retracts onto the union of all the critical submanifolds 
$C$, for which $f(C)<t_1$.

Observe that
our assumptions imply that for every critical submanifold $C$ of $f$ the cohomology group $H^{1}(C;\Z_{2})=\Hom(H_1(C;\Z), \Z_2)$ 
vanishes. Thus the unstable bundle of $C$ is orientable and, using excision and the Thom isomorphism, one obtains
$$H_{j}(M_{i+1},M_i;\Z) \simeq \bigoplus_{f(C) \in (t_i, t_{i+1})} H_{j-\ind_{f}(C)} (C;\Z).$$ 
In particular, we see that the non-vanishing groups $H_{j}(M_{i+1},M_i;\Z)$ are torsion free and are concentrated in degrees $j$ which are multiples of $k$. 

Arguing inductively, we can now assume
that the non-vanishing groups $H_{*}(M_i;\Z)$ are all in degrees which are multiples of $k$. Then from the long exact sequence of the pair
$(M_{i+1},M_i)$ one obtains
\begin{eqnarray}\label{isorel}
H_{*}(M_{i+1};\Z) &\simeq& H_{*}(M_i;\Z) \oplus H_{*}(M_{i+1},M_i;\Z) \\
&\simeq& H_{*}(M_i;\Z) \oplus \bigoplus_{f(C) \in (t_i, t_{i+1})} H_{*-\ind_{f}(C)} (C;\Z).\nonumber\end{eqnarray}
Combining these isomorphisms for $i=1, \dots, k$ we get the isomorphism (\ref{iso}) for $i+1$. This completes the proof. 
\end{proof}

Next we give sufficient conditions for a collection of homology classes represented by a collection of submanifolds to form a 
 basis in homology. 

\begin{proposition} \label{prop1} Suppose that additionally to the assumptions of Proposition \ref{Lac}, for each critical submanifold $C\subset M$ of $f$ we are given a closed 
submanifold 
$W_C\subset M$ and a finite collection of closed submanifolds $\W_C=\{Z; Z\subset W_C\}$ such that the following conditions are satisfied:
 \begin{enumerate}
   \item $C\subset W_C$, and $\dim W_C= \ind_f(C)+\dim C$,
   \item The function $f|W_C$ is nondegenerate in the sense of Bott and 
  achieves its maximum on $C$, 
   \item Each $Z\in \W_C$ is transversal to $C$ as a submanifold of $W_C$, 
   \item The set of homology classes    
$[Z\cap C]\in H_\ast(C; \Z_2)$, for all $Z\in \W_C$, forms a basis of $H_\ast(C; \Z_2)$. 
\end{enumerate} 
Then the set of the homology classes $[Z]\in H_\ast(M;\Z_2)$, for all $Z\in \W_C$ and for all critical submanifolds $C\subset  {\rm {Crit}}(f)$, forms a basis of $H_\ast(M;\Z_2)$. 
\end{proposition}

\begin{proof} 
We will use the notations introduced in the beginning of the proof of Proposition 
\ref{prop1}. 

Our inductive statement is that for $i=1, 2, \dots, k$ the set of homology classes $[Z]\in H_\ast(M_i;\Z_2)$ for all $Z\in \W_C$ and for all ctirical submanifolds $C\in {\rm {Crit}}(f)$ satisfying 
$f(C)<t_i$ forms a basis of $H_\ast(M_i;\Z_2)$. 

We know that the manifold $M_1$ deformation retracts onto the disjoint union of the critical submanifolds $C$ on which $f$ achieves its minimum. For such a $C$ we have 
$W_C=C$ and the family $\W_C=\{Z; Z\subset W_C\}$ is such that the classes $[Z]\in H_\ast(C;\Z_2)$ form a basis of $H_\ast(C;\Z_2)$. Hence the set of all homology classes $[Z]\in H_\ast(M_1;\Z_2)$ 
(where $Z\in \W_C$ and $C$ is a critical submanifold of $f$ satisfying $f(C)<t_1$) forms a basis. This proves our inductive statement for $i=1$. 

Assume that the inductive statement is true for $i$ and consider its validity for $i+1$. We know that the classes $[Z]\in H_\ast(M_i;\Z_2)$ for all $Z\in \W_C$ and for all $C$ with 
$f(C)<t_i$ form a basis of $H_\ast(M_i;\Z_2)$. In view of isomorphism (\ref{isorel}) it is enough to show that the set of relative homology classes 
$[(Z, Z\cap M_i)]\in H_\ast(M_{i+1}, M_i;\Z_2)$ is a basis of $H_\ast(M_{i+1}, M_i;\Z_2)$, where $Z\in \W_C$ and $C$ runs over all critical submanifolds $C$ of $f$ with $f(C)\in (t_i, t_{i+1})$. Using a combination of excision, deformation retraction and the Thom isomorphism one obtains an isomorphism 
$$H_\ast(M_{i+1}, M_i;\Z_2) \simeq \bigoplus_{f(C)\in (t_i, t_{i+1})} H_{\ast-\ind_f(C)}(C;\Z_2)$$
and, as is well-known, under this isomorphism the image of the homology class $[(Z, Z\cap M_i)]$ equals $[Z\cap C]\in H_{\ast - \ind_f}(C;\Z_2)$. Thus we see that our assumptions imply that the relative homology classes $[(Z, Z\cap M_i)]\in H_\ast(M_{i+1}, M_i;\Z_2)$ form a basis of $H_\ast(M_{i+1}, M_i;\Z_2)$ confirming the induction step.  

This completes the proof. 
\end{proof}

Next we give a version of Proposition \ref{prop1} with integral coefficients.

\begin{proposition}\label{prop2} Suppose that additionally to the assumptions of Proposition \ref{prop1}, each of the submanifolds $Z\in \W_C$ is oriented. 
Fix an orientation of the normal bundle to $C$ in $W_C$. Then each intersection $Z\cap C$ is canonically oriented and the symbol $[Z\cap C]\in H_\ast(C;\Z)$ will denote the homology class of $C$ realized by $Z\cap C$. Assume that for each critical submanifold $C\subset {\rm {Crit}}(f)$ the collection of classes $[Z\cap C]\in H_\ast(C;\Z)$,
where $Z\in \W_C$, forms a basis of $H_\ast(C;\Z)$. Then the collection of the homology classes $[Z]\in H_\ast(M;\Z)$, for all $Z\in \W_C$ and for all critical submanifolds $C\subset  {\rm {Crit}}(f)$, forms a basis of $H_\ast(M;\Z)$. 
\end{proposition}

The proof of Proposition \ref{prop2} is similar to that of Proposition \ref{prop1} and is omitted.

\section{Basic facts about the spaces of polygons $E_d(\ell)$}

We collect in this section some basic facts about $E_d(\ell)$. 
Versions of these statements for other closely related spaces are well-known and are provable by similar arguments.  

First we show that for a generic length vector $\ell\in \R^n_>$ the variety $E_d(\ell)$ is a closed smooth manifold of dimension $d(n-1) -n$. 

Consider the map $F: (\R^{d})^{n-1} \rightarrow \R^{n}$ given by the formula
$$(v_{1}, \dots, v_{n-1}) \mapsto (|v_{1}|, |v_{2}-v_{1}|, \dots, |v_{n-1}-v_{n-2}|, |v_{n-1}|).$$
One clearly has $E_{d}(\ell)=F^{-1}(\ell)$ as the following picture explains. 
\begin{figure}[h]
\begin{center}
\resizebox{7cm}{5cm}{\includegraphics[68,431][429,718]{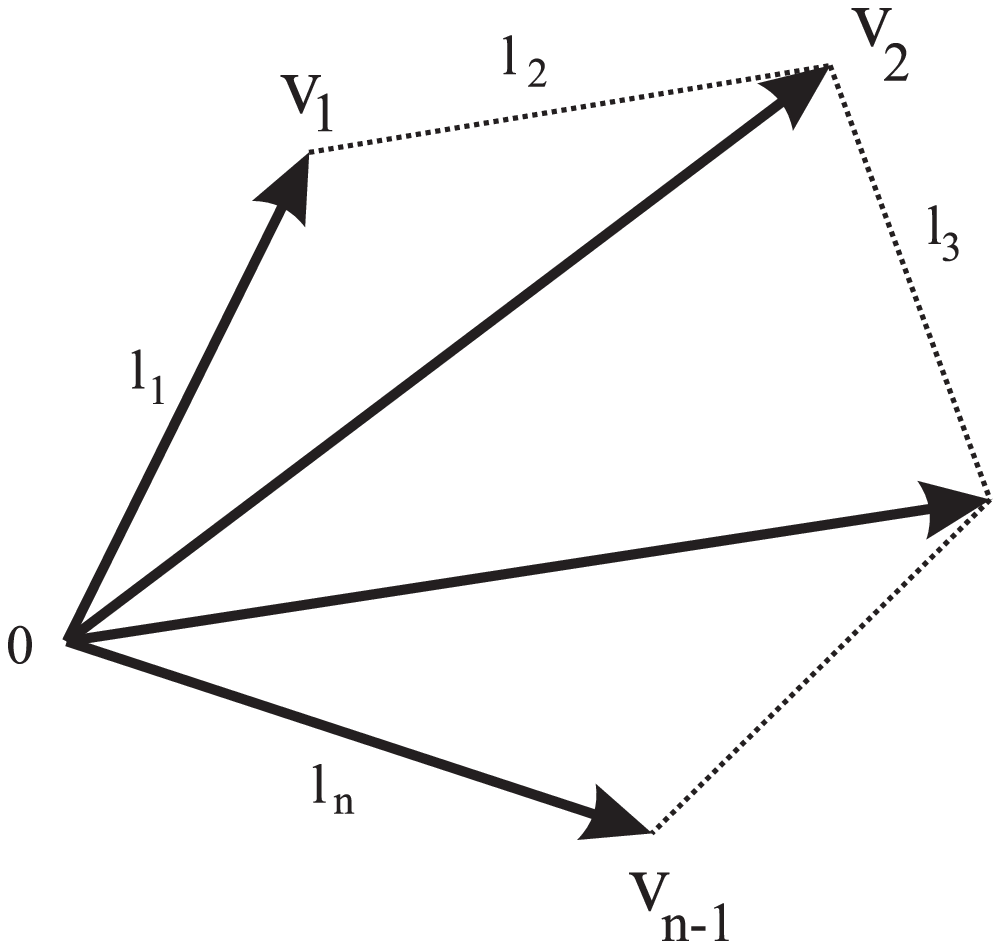}}
\end{center}
%\caption{}\label{fig1}
\end{figure}
The map $F$ is smooth on the open subset $\Omega\subset (\R^d)^{n-1}$ given by the inequalities $v_1\not=0$, $v_j\not= v_{j+1}$ for $j=1, \dots, n-2$ and $v_{n-1}\not=0$. 
Clearly $E_d(\ell)$ is contained in $\Omega= F^{-1}(\R^n_>)$.

\begin{proposition}\label{gen}
A vector $\ell\in \R^n_>$ is a regular value of $F|\Omega$ if and only if $\ell$ is generic, i.e. $\sum_{i=1}^n \epsilon_il_i\not=0$ for $\epsilon_i=\pm 1$. 
Thus, for a generic $\ell\in \R^d_>$ 
the preimage $F^{-1}(\ell) =E_d(\ell)$ is a smooth closed manifold of dimension $d(n-1) -n$. 
\end{proposition}

\begin{proof} It is obvious that 
Proposition \ref{gen} follows from the statement that critical points of the map $F|\Omega$ are the tuples $v=(v_{1}, \dots, v_{n-1})\- \in \Omega$ such that 
the vectors $v_{1}, \dots, v_{n-1}$ $\in \R^{d}$ are collinear.

%Since for a length vector $l$ the preimage $F^{-1}(l)$ contains points $v$ of this form only if $l$ is non-generic, this will imply that generic length vectors are regular values of $F$.  \\ \\

Let $V=(V_{1}, \dots, V_{n-1}) \in T_{v}(\Omega)$ be a tangent vector. 
%
%We denote $$u_{1}=\frac{1}{l_{1}}v_{1}, \text{ } u_{n}=-\frac{1}{l_{n}}v_{n-1}$$ $$u_{j}=\frac{1}{l_{j}}(v_{j}-v_{j-1}), \text{ }2 \leq j \leq n-1$$ and write 
We denote by $\langle \cdot, \cdot \rangle$ the standard scalar product on $\R^{d}$. 
The derivative $D_{V}F$ of $F$ in direction $V$ is a vector of the form $((D_VF)_1, \dots, (D_VF)_n)\in \R^n$ where
$$(D_{V}F)_{1}=\langle V_{1},\frac{v_{1}}{|v_1|}\rangle, \quad (D_{V}F)_{n}=\langle V_{n-1},\frac{v_{n-1}}{|v_{n-1}|}\rangle$$ 
and for $j=2, \dots, n-1$ one has
$$(D_{V}F)_{j}=  \langle V_{j}-V_{j-1}, \frac{v_{j}-v_{j-1}}{|v_j-v_{j-1}|} \rangle.$$%, \text{ }j=2, \dots, n-1$$ $$(T_{V}F)_{n}=-\langle V_{n-1},u_{n} \rangle$$

Consider the unit vectors $u_1, \dots, u_n \in S^{d-1}$ given by
 $$u_{1}=\frac{v_1}{|v_1|}, \text{ } u_{n}=-\frac{v_{n-1}}{|v_{n-1}|}$$ 
and 
$$u_{j}=\frac{v_{j}-v_{j-1}}{|v_{j}-v_{j-1}|}$$
for $j=2, \dots, n-1$. 
Then for any $j=1, \dots, n$ we may write $(D_VF)_j = \langle V_j-V_{j-1}, u_j\rangle$ where we understand $V_0=0=V_n$.

The point $v$ is a critical point of $F$ iff the differential of $F$ at $v$ has rank less than $n$. 
This is equivalent to the existence of a nonzero vector $a=(a_{1}, \dots, a_{n}) \in \R^{n},$ so that the image of the derivative of $F$ at $v$ is contained in the hyperplane orthogonal to $a$. 
We conclude that
for all $V_{1}, \dots, V_{n-1}$ one would have
$$\sum \limits_{j=1}^{n-1} \langle V_{j},a_{j}u_{j}-a_{j+1}u_{j+1} \rangle=0$$ (again, understanding that $V_0=0=V_n$.)
This happens iff the unit vectors $u_{1}, \dots, u_{n}$ satisfy $a_{j}u_{j}=a_{j+1}u_{j+1}$ for $j=1, \dots, n-1$. 
Taking the absolute value we find that $|a_i|=|a_{i+1}|\not=0$ and hence
%we may assume that $a_i=\pm 1$. Thus we see that 
the vectors $u_{1}, \dots, u_{n}$, and therefore also $v_{1}, \dots, v_{n-1}$, are collinear. 
\end{proof}

\begin{proposition}\label{prop32} If two generic length vectors $\ell, \ell'\in \R^n_>$ lie in the same chamber then the manifolds $E_d(\ell)$ and $E_d(\ell')$ are ${\rm O}(d)$-equivariantly diffeomorphic.
\end{proposition}
\begin{proof}
Suppose that $\ell$ and $\ell'$ lie in the same chamber. Then for $t\in [0,1]$ the vector $\ell^t=(1-t)\ell+t\ell'$ is generic. 
Then, applying the previous Proposition, we see that $F$ is a submersion on the interval $I=\{\ell^t; t\in [0,1]\}$ and hence $E(\ell, \ell') = F^{-1}(I)$ is a smooth manifold with boundary $E_d(\ell) \sqcup E_d(\ell')$. This cobordism is trivial since the projection $E(\ell, \ell') \to I$ is a smooth function with no critical points. Therefore the boundary components
$E_d(\ell)$ and $E_d(\ell')$ are diffeomorphic. 
The map $F$ is invariant with respect to the diagonal action of the orthogonal group ${\rm O}(d)$ on $(\R^{d})^{n-1}$. We conclude that
the projection $\pi: E(\ell,\ell') \rightarrow I$ is ${\rm O}(d)$-invariant. Consider the flow of the gradient of $\pi$ with respect to the metric
on $E(\ell,\ell')$ induced by the Euclidean metric of $(\R^{d})^{n-1}$. The flow defines a diffeomorphism $E_{d}(\ell) \simeq E_{d}(\ell')$. Since both
$\pi$ and the Euclidean metric are ${\rm O}(d)$-invariant, so is this diffeomorphism. 
\end{proof}

\section{Cohomology of $E_d(\ell)$}

The crucial step in the proof of Theorem \ref{Wal} is a computation of the subring $H^{(d-1)*}(E_{d}(\ell);\Z_{2}) \subset H^{*}(E_{d}(\ell);\Z_{2})$ which is performed in this section. We always suppose 
that $d>2$; in the case $d=2$ the main results are either known or follow easily from the known results. Besides, the techniques used in the proofs below
work only under the assumption $d>2$, while in the case $d=2$ different techniques can be applied.
It will be convenient to assume that the length vector $\ell=(l_1, \dots, l_n)$ is {\it ordered}, i.e. $$l_{1} \leq l_2\le \dots \leq l_{n}.$$ 
Since the diffeomorphism type of the space $E_d(\ell)$ does not change if one permutes the coordinates of $\ell$, this assumption is not restrictive: any length vector becomes ordered after a permutation.

\begin{definition} 
We say that a subset $J \subset \{1, \dots, n\}$ is {\it long} with respect to a length vector $\ell\in \R^n_>$ if $\sum \limits_{j \in J} l_{j}-\sum \limits_{j \notin J}l_{j}>0.$ 

A subset $J \subset \{1, \dots, n\}$ is {\it short} with respect to $\ell$ if $\sum \limits_{j \in J} l_{j}-\sum \limits_{j \notin J}l_{j}<0.$

A subset $J \subset \{1, \dots, n\}$ is {\it median} with respect to $\ell$ if $\sum \limits_{j \in J} l_{j}=\sum \limits_{j \notin J}l_{j}.$
\end{definition}

\begin{proposition} \label{ECohomology}
For any ordered length vector $\ell\in \R^n_>$ and $d>2$ there is an isomorphism of graded rings $$H^{(d-1)*}(E_{d}(\ell);\Z_{2}) \simeq \Lambda_{d}(Z_{1}, \dots, Z_{n})/I,$$ 
where $\Lambda_{d}(Z_{1}, \dots, Z_{n})$ is the exterior algebra with coefficients in $\Z_2$ on generators $Z_1, \dots, Z_n$ having degree $d-1$ and 
$I \subset \Lambda_d(Z_{1}, \dots, Z_{n})$ is the ideal generated (as an ideal) by the monomilas $Z_{j_{1}} \dots Z_{j_{k}}$ such that $1\le j_1<j_2< \dots <j_k<n$ and the 
set $\{j_1, j_2, \dots, j_k, n\}$ is long with respect to the length vector $\ell$. 
\end{proposition} 

We emphasize that in Proposition \ref{ECohomology} we do not assume that $\ell$ is generic. 

The proof of Proposition \ref{ECohomology} will use a series of Lemmas given below. 

Denote $W=(S^{d-1})^{n}$ and consider the function $f_{\ell}: W \rightarrow \R$ given by 
\begin{eqnarray}\label{dif}
f_{\ell}\colon (u_{1}, \dots, u_{n}) \mapsto -\left|\sum \limits_{j=1}^{n}l_{j}u_{j}\right|^{2}, \quad u_j\in S^{d-1}.\end{eqnarray} 
Note that $E_{d}(\ell)=f_{\ell}^{-1}(0)$ is a critical submanifold (the set of points where $f_\ell$ achieves its maximum). Let us compute the remaining critical points of 
$f_{\ell}$.
For a subset $J \subset \{1, \dots n\}$, let $P_{J} \subset W$ denote the submanifold 
$$P_{J}=\{(u_{1}, \dots, u_{n}): u_{i}=u_{j}=-u_{k} \text{ for all } i,j \in J, k \notin J\}$$ of $W$. Note that each $P_{J}$ is diffeomorphic to the sphere $S^{d-1}$ and $P_J=P_{\bar J}$ for the complement $\bar J$ of $J$.
 
\begin{lemma}\label{CritPoints}
The restriction of $f_{\ell}$ to $W-E_{d}(\ell)$ is non-degenerate in the sense of Bott. 
The set of critical points of $f_\ell$ in $W-E_{d}(\ell)$
is the union of all submanifolds $P_{J}$ where $J$ is a long subset with respect to $\ell$. 
The Morse-Bott index of the critical submanifold $P_{J}$ equals
$$\ind_{f_{\ell}}(P_{J})= (d-1)(n-|J|),$$
where $|J|$ denotes the cardinality of $J$. 
\end{lemma}
\begin{proof}
Let  $\langle \cdot, \cdot \rangle$ denote the standard scalar product in the Euclidean space $\R^d$. 
For $u \in S^{d-1}$, we will identify the tangent space 
$$T_{u}S^{d-1} = \{v \in \R^{d}: \langle u, v \rangle = 0\}.$$ 
Let $(v_{1}, \dots, v_{n}) \in T_{(u_{1}, \dots, u_{n})} (S^{d-1})^{n}$ be a tangent vector. 
For $j=1, \dots, n$ the differential of the function $f_{\ell}$ (given by (\ref{dif})) in the direction $v_{j}$ is 
$$D_{v_{j}}f_{\ell}=-2 \langle l_{j}v_{j}, \, \sum \limits_{i=1}^{n}l_{i}u_{i} \rangle.$$ 
At a critical point of $f_{\ell}$, we must have $D_{v_{j}}f_{\ell}=0$ for all $v_{j} \in T_{u_{j}} S^{d-1}$ and $j=1, \dots, n$. 
Hence the vectors $u_{j}$ and $\sum \limits_{i=1}^n l_{i}u_{i}\not=0$ are collinear. Since this is true for every $j$, we find
that the critical points of $f_\ell$ lying in $W-E_d(\ell)$ coincide with the set of tuples $(u_{1}, \dots, u_{n})\in W$ such that $u_{j}=\pm u_{k}$ for all $j,k \in \{1, \dots, n\}$. 
This set coincides with the union of all submanifolds $P_J$ with long subset $J$. %Indeed, one has $P_{\bar J}=P_J$ for the complement $\bar J$ of $J$, and thus it is enough to 

Let us fix an element $p \in S^{d-1}$ and consider for every long subset $J \subset \{1, \dots, n\}$ the point $p_{J} \in P_{J}$ with $u_{j}=p$ for $j \in J$ and $u_{j}=-p$ for $j \notin J$. We want to compute explicitly the Hessian of $f_{\ell}$ at the critical point $p_{J}$. By symmetry, we may assume $p=e_{1}$. We parametrize $S^{d-1}$ in a neighbourhood of $e_{1}$ by the map
$\R^{d-1} \rightarrow S^{d-1} \subset \R^{d}$ given by 
$$(r_{2}, \dots, r_{d}) \mapsto \frac{1}{(1+\sum_{2 \leq j \leq d} r^{2}_{j})^{1/2}}(e_{1}+\sum_{2 \leq j \leq d} r_{j}e_{j})$$
and in a neighbourhood of $-e_{1}$ by the map $\R^{d-1} \rightarrow S^{d-1} \subset \R^{d}$ where 
$$(r_{2}, \dots, r_{d}) \mapsto -\frac{1}{(1+\sum_{2 \leq j \leq d} r^{2}_{j})^{1/2}}(e_{1}+\sum_{2 \leq j \leq d} r_{j}e_{j}).$$
Then at the point $p_{J}$ we may express the second derivatives of $f_{\ell}$ as:
$$\frac{\partial^{2}}{\partial r^{(j_{1})}_{k_{1}} \partial r^{(j_{2})}_{k_{2}}}f_{\ell}= 
\begin{cases}  -2l^{2}_{j}+2\epsilon_{J}(j)l_{j}L_{J}, & \text{ if } (k_{1},j_{1})=(k_{2},j_{2})=(k,j), \\ -2l_{j_{1}}l_{j_{2}},& \text{ if } k_{1}=k_{2}, j_{1} \neq j_{2},\\ 
0, & \text{ if } k_{1} \neq k_{2}. \end{cases}$$ 
Here $k_{1}, k_{2} \in \{ 2, \dots, d \}$, $j_{1},j_{2} \in \{ 1, \dots, n \}$ and $L_{J}=\sum_{j \in J} l_{j}-\sum_{j \notin J}l_{j}.$ 
We have also denoted
$$\epsilon_{J}(j)= \begin{cases} +1 & j \in J, \\ -1 & j \notin J. \end{cases}$$
It follows that the Hessian of $f_{\ell}$ at $p_{J}$ is congruent to a matrix $A$ of size  ${(d-1)n \times (d-1)n}$ of the following form. Let $D \in \R^{n \times n}$ be the diagonal matrix whose $j$-th diagonal entry for $j=1, \dots, n$ is %$$-\frac{\epsilon_{J}(j)}{l_{j}}L_{J}$$ 
$-\epsilon_{J}(j) \cdot L_{J}\cdot l_{j}^{-1}$ 
and $E \in \R^{n \times n}$ be the matrix with all entries equal to $1$. Then $A$ is obtained from the difference $D-E$ by replacing every entry $\lambda$ with the matrix $\lambda I_{d-1}$ where $I_{d-1} \in \R^{(d-1) \times (d-1)}$ is the identity matrix. We conclude that the Hessian of $f_{\ell}$ has the same eigenvalues as the matrix $D-E$ and the multiplicity of every eigenvalue is multiplied by $d-1$. Using the computation in
\cite{Fa5}, Lemma 1.4, the index of the Hessian is $(n-|J|)(d-1)$ and the multiplicity of the zero eigenvalue is $(d-1)$. Since $d-1$ is also the dimension of the critical submanifold, $f_{\ell}$ is a Morse-Bott function. 
\end{proof}

We now construct a homology basis for the complement $W-E_{d}(\ell)$. 
Fix a point $p \in S^{d-1}$ and for every subset $J \subset \{1, \dots, n\}$ define submanifolds $V_{J},W_{J} \subset W$ as follows: 
$$ V_{J}= \{ (u_{1}, \dots, u_{n}) \in W \colon \text{ } u_{i}=p \text{ for } i \in J \}$$ 
and
$$W_{J}= \{ (u_{1}, \dots, u_{n}) \in W \colon \text{ }u_{i}=u_{j} \text{ for }i,j \in J \}.$$
Note that $V_J\subset W_J$ and the homology class $[V_{J}] \in H_{(d-1)(n-|J|)}(W;\Z_2)$ is independent of the choice of the point $p \in S^{d-1}$.

\begin{proposition}\label{prop44} If $J\subset \{1, \dots, n\}$ is long then the submanifold $W_J$ lies in the complement $W-E_d(\ell)$. 
The classes $[W_{J}], [V_{K}]$, where $J$, $K$ run over all long subsets of $\{1, \dots, n\}$ 
with $|J|=n-k+1$ and $|K|=n-k$, form a basis of the homology vector space $H_{(d-1)k}(W - E_{d}(\ell) ; \Z_2)$.  
\end{proposition}

We emphasize that in this statement we do not require the length vector $\ell$ to be generic. 

\begin{proof} Let $a<0$ be such that 
$$-(\sum_{i\in J} l_i - \sum_{i\notin J}l_i)^2<a<0$$
for any long subset $J\subset \{1, \dots, n\}$. The preimage $W^a=f_\ell^{-1}(-\infty, a]$ is a compact manifold with boundary and the inclusion $W^a\subset W-E_d(\ell)$ is a homotopy equivalence since all critical points of $f_\ell$, which lie  in $W-E_d(\ell)$, lie in fact in $W^a$. 

By Lemma \ref{CritPoints} the restriction of $f_{\ell}$ onto $W^a$ is a Morse-Bott function with the critical submanifolds 
$P_J$ (labeled by long subsets $J$) which all are $(d-1)$-dimensional spheres, and thus each $P_J$ is $(d-1)$-lacunary. 
We claim that all conditions of the Morse-Bott lacunary principle as described in \S \ref{sec:lac} are satisfied. 
Indeed, all Morse-Bott indices are multiples of $d-1$ by 
Lemma \ref{CritPoints}. Moreover, the function $f_\ell|W_J$ achieves its maximum at $P_J\subset W_J$. If we denote by $\W_J=\{V_J, W_J\}$ the family consisting of two submanifolds then the assumptions of Proposition \ref{prop1} are satisfied; in particular the homology classes of the intersections $W_J\cap P_J=P_J$ and 
$V_J\cap P_J=\{\ast\}$ form a basis of $H_\ast(P_J;\Z_2)$. Thus, by Proposition \ref{prop1}, the homology classes 
$$[V_J]\in H_{(d-1)(n-|J|)}(W-E_d(\ell);\Z_2)$$ and $$[W_J]\in H_{(d-1)(n-|J|+1)}(W-E_d(\ell);\Z_2),$$ where $J$ runs over all long subsets, form a basis of the homology group
$H_\ast(W-E_d(\ell);\Z_2)$. 
\end{proof}

Clearly, a basis of $H_{(d-1)k}(W;\Z_2)$ is given by the collection of classes $[V_{K}]$ where $K \subset \{1, \dots, n\}$ is an arbitrary subset with $|K|=n-k$. 
It is easy to see that for $|K|+|K'|=n$
the intersection numbers in $W$ of these classes are given by the formula
\begin{eqnarray}\label{intnumber}
[V_{K}] \cdot [V_{K'}]= \begin{cases}  1, & \text{if } K \cap K'= \emptyset ,\\ 0 & \text{if}\,  K\cap K'\not=\emptyset. \end{cases}
\end{eqnarray}

We will also need the intersection numbers in $W$ involving the classes $[W_J]\in H_{(d-1)(n-|J|+1)}(W;\Z_2)$. Note that $[W_J]$ and $[V_K]$ have complementary dimension in 
$W$ if $|J|+|K|=n+1$. Besides, the classes $[W_J]$ and $[W_K]$ have complementary dimension in $W$ if $|J|+|K|=n+2$.

\begin{lemma} \label{IntLemma}
(A) For $J,K \subset \{1, \dots, n\}$ with $|J|+|K|=n+1$ the intersection number of the homology classes $[W_{J}]$ and $[V_{K}]$ in $W$ is 
 $$[W_{J}] \cdot [V_{K}]= \begin{cases} 1, & \text{if }\,  |J \cap K|= 1, \\ 0, & \text{if}\, \,  |J \cap K| \not= 1. \end{cases}$$
(B) For every pair of subsets $J, K \subset \{1, \dots, n\}$ satisfying $|J|+|K|= n+2$ the intersection number $[W_{J}] \cdot [W_{K}]\in \Z_2$ in $W$ is trivial\footnote{Many statements made in this section hold essentially without changes with integral coefficients. However, as the proof below shows, over integers the intersection number 
$[W_{J}] \cdot [W_{K}]$
 can be nonzero even.}.
\end{lemma}
\begin{proof}
Let $|J \cap K| > 1$. Fix $k \in K\cap J$ and $p' \in S^{d-1}$ with $p \neq p'$. Define a submanifold $V'_{K} \subset W$ by $$V'_{K}=\{(u_{1}, \dots, u_{n}) \in W: u_{j}=p \text{ for } j \in K -\{k\} \text{ and } u_{k}=p'\}$$
Then $[V_{K}]=[V'_{K}]$ and $W_{J} \cap V'_{K}= \emptyset$. We conclude that $[W_{J}] \cdot [V_{K}]=0$. \\ \\  
If $|J \cap K|=1$ then $W_{J}$ and $V_{K}$ have a unique point of intersection given by $u_{j}=p$ for $j=1, \dots, n$. It is not difficult to see that the intersection is transverse, compare \cite{Fa2}. This proves (A). 

To demonstrate the claim (B) note that $|J \cap K|\ge 2$ and let $j \in J \cap K$. Fix a diffeomorphism $ \varphi: S^{d-1} \rightarrow S^{d-1}$ which is homotopic to the identity and has no fixed points (if $d$ is even) or has two non-degenerate fixed points (if $d$ is odd). Define 
$\Phi: W=(S^{d-1})^{n} \to W=(S^{d-1})^{n}$ as the map which applies 
$\varphi$ to the $j$-th factor and is identical on all other factors. Then $[\Phi(W_{J})]=[W_{J}]$ and the 
submanifolds $\Phi(W_{J})$ and $W_{K}$ are either disjoint or intersect transversally in exactly two points. Hence the {\it mod} 2 intersection number
$[W_J]\cdot [W_K]$ is trivial. 
%In the case $|J \cap K|>2 $ an analogous argument shows $[W_{J}] \cdot [W_{K}]=0$.
\end{proof}

We will now use Lemma \ref{IntLemma} to pass to a new basis of $H_{(d-1)k}(W;\Z_2)$ which will be more convenient for the proof of Proposition \ref{ECohomology}. 

\begin{lemma}\label{lm46}
For $0 \leq k \leq n$, the homology classes $[W_{J}]$ and $[V_{K}]$ with %$J, K \subset \{1, \dots, n\}$, 
$|J|=n-k+1$,  $|K|=n-k$ and $n \in J$, $n \in K$ form a basis 
of $H_{(d-1)k}(W;\Z_2)$. 
\end{lemma}
\begin{proof} Since the number of the specified elements 
$\binom {n-1} {k-1} + \binom {n-1} k$ 
coincides with the dimension $\binom n k$ of 
$H_{(d-1)k}(W; \Z_2)$, it suffices to check that the elements of Lemma \ref{lm46} are linearly independent. For this purpose we will construct a dual system (with respect to the intersection form) as follows.
For a subset $I\subset \{1, \dots, n\}$ containing the index $n$ we define $\hat I= \bar I\cup \{n\}$, where $\bar I$ denotes the complement of $I$. 
Then the following system forms a dual basis to the system of element specified in Lemma \ref{lm46}:
$$[W_J]^\ast= [V_{\hat J}], \quad [V_K]^\ast= [W_{\hat K}].$$
This follows directly from formula (\ref{intnumber}) and Lemma \ref{IntLemma}. The existence of a dual system implies that the system of Lemma \ref{lm46} is linearly independent. 

This completes the proof.\end{proof}
%Let $H'_{k} \subset H_{(d-1)k}(W;\Z_2)$ denote the subgroup generated by the classes $[W_J]$ and $[V_K]$ described above. 

%the intersection form of $W$ restricts to a non-degenerate bilinear form 
%$$H'_{k} \times H_{(d-1)(n-k)}(W; \Z_2) \rightarrow \Z_2.$$
%As we know the group $H_{(d-1)(n-k)}(W; \Z_2)$ has as a basis the classes $[V_{K'}]$ where $K' \subset \{1, \dots, n\}$ is any subset with $|K'|=k$. 
%For every such class $[V_{K'}]$, there is a unique specified generator of $H'_{k}$ which has non-vanishing intersection number with $[V_{K'}]$. Namely, if $n \in K'$ then for $J=\{1, \dots, n\}-K' \cup \{n\}$ we have $[W_{J}] \cdot [V_{K'}]= 1$ and the intersection number of $[V_{K'}]$ with every other generator of $H_k'$ vanishes. 
%If $n \notin K'$ then we have $[V_{K}] \cdot [V_{K'}]= 1$ for $K= \{1, \dots, n\}-K'$ and all the other intersection numbers with generators of $H_k'$ vanish, as follows from Lemma \ref{IntLemma}.

%
%\begin{corollary}\label{cordual} The basis of $H_{(d-1)(n-k)}(W;\Z_2)$ dual to the basis of Lemma \ref{lm46} is given by
%$$[W_J]^\ast = [V_{\bar J\cup \{n\}}], \quad [V_K]^\ast = [V_{\bar K}],$$
%where $\bar J$ and $\bar K$ denote the complements of $J$ and $K$ in $\{1, \dots, n\}$. 
%\end{corollary}

The proof of Proposition \ref{ECohomology} relies on the study of the homomorphism 
\begin{eqnarray}\label{jk}
j_{k} \colon H_{(d-1)k}(W-E_{d}(\ell);\Z_{2}) \to H_{(d-1)k}(W;\Z_{2})
\end{eqnarray}
induced by the inclusion. %We assume that the length vector $\ell$ is generic and hence avery subset of $\{1, \dots, n\}$ is either long or short. 
We will use the basis of $H_{(d-1)k}(W-E_{d}(\ell);\Z_{2})$ given by Proposition \ref{prop44} and the basis of $H_{(d-1)k}(W;\Z_{2})$ given by Lemma \ref{lm46}. 
We shall split $$H_{(d-1)k}(W-E_{d}(\ell);\Z_{2}) = A_{k} \oplus A'_{k} \oplus B_{k} \oplus B'_{k},$$ where: 
\begin{itemize}  
\item  $A_{k}$ (resp $A'_{k}$) is generated by classes $[W_{J}]$ with $|J|=n-k+1$, $J$ long and $n \in J$ (resp. $n \notin J$). 

\item $B_{k}$ (resp. $B'_{k}$) is generated by classes $[V_{K}]$ with $|K|=n-k$, $K$ long and $n \in K$ (resp. $n \notin K$). 
\end{itemize}

Let us also split $$H_{(d-1)k}(W;\Z_{2})=A_{k} \oplus B_{k} \oplus C_{k} \oplus D_{k},$$
where $A_{k}$ and $B_{k}$ are as above and
\begin{itemize}

\item $C_{k}$ is generated by the classes $[W_{J}]$ with $|J|=n-k+1$, $n \in J$ and $J$ is short or median.

\item $D_{k}$ is generated by the classes $[V_{K}]$ with $|K|=n-k$, $n \in K$ and $K$ is short or median. 
\end{itemize}

The map $j_{k}$ clearly restricts to the identity on $A_{k}$ and on $B_{k}$.
 
\begin{lemma} One has \begin{eqnarray}\label{partone}
j_{k}(A'_{k}) \subset A_{k}\end{eqnarray}
and 
\begin{eqnarray}\label{parttwo}
j_{k}(B'_{k}) \subset A_{k}\oplus B_k.
\end{eqnarray}
\end{lemma}
\begin{proof} Let $[W_{I}] \in A'_{k}$, i.e. $|I|=n-k+1$, $I$ is long and $n \notin I$. Using the dual basis described in the proof of Lemma \ref{lm46} 
we obtain the following Fourier decomposition
$$j_k([W_I]) = \sum_{[W_J]} \left([W_I]\cdot [V_{\hat J}]\right) [W_J] + \sum_{[V_K]} \left([W_I]\cdot [W_{\hat K}]\right)[V_K],$$
where $J$ and $K$ run over subsets of $\{1, \dots, n\}$ containing $n$ with $|J|=|I|$ and $|K|=|I|-1$. 

The coefficients $[W_I]\cdot [W_{\hat K}]$ in the second sum vanish according to statement (B) of Lemma \ref{IntLemma}. 
We claim that the coefficient $[W_I]\cdot [V_{\hat J}]$ of the first sum also vanishes assuming that $J$ is short and $n\in J$. 
Indeed, if this coefficient is nonzero then by Lemma \ref{IntLemma} $|I\cap \hat J|=1$ which means that $J$ is obtained from 
$I$ by removing one of its elements and adding $n$; but this is inconsistent with the assumptions that $I$ is long, $J$ is short and $l_n \ge l_i$ for any $i=1, \dots, n$. 
This proves (\ref{partone}). 

To prove (\ref{parttwo}), consider a long subset $L\subset \{1, \dots, n\}$ with $n\notin L$. As above we have the Fourier decomposition 
$$j_k([V_L]) = \sum_{[W_J]} \left([V_L]\cdot [V_{\hat J}]\right) [W_J] + \sum_{[V_K]} \left([V_L]\cdot [W_{\hat K}]\right)[V_K],$$
where $J$ and $K$ run over subsets of $\{1, \dots, n\}$ containing $n$ with $|J|=|L|+1$ and $|K|=|L|$. 

The coefficient $[V_L]\cdot [V_{\hat J}]$ of the first sum vanishes if $J$ is short. Indeed, since by (\ref{intnumber}), if the intersection number $[V_L]\cdot [V_{\hat J}]\not=0$ 
then $L=J-\{n\}$ implying that $L$ is also short which contradicts our assumption. 

The coefficient
$[V_L]\cdot [W_{\hat K}]$ in the second sum vanishes if $K$ is short. Indeed, if this intersection number is nonzero then by Lemma \ref{IntLemma}
$|L\cap \hat K|=1$ implying that $K$ is obtained from $L$ by removing one of the elements and adding $n$. Since $L$ is long and $l_n \ge l_i$ for all $i=1, \dots, n$ it follows that 
$K$ must be long as well. This proves (\ref{parttwo}).  \end{proof}

\begin{corollary} \label{Image}
The image of $j_{k}$ (see (\ref{jk})) is generated by the homology classes $[W_{J}]$, $[V_{K}]$ where $J$ and $K$ are long subsets containing $n$ with $|J|=n-k+1$ and $|K|=n-k$. 
\end{corollary}

We are now ready to prove Proposition \ref{ECohomology}.

{\it Proof of Proposition \ref{ECohomology}}. 
Consider the cohomological long exact sequence
$$H^{k(d-1)}(W;\Z_{2}) \rightarrow H^{k(d-1)}(E_{d}(\ell);\Z_{2}) \rightarrow H^{k(d-1)+1}(W,E_{d}(\ell);\Z_{2}).$$

By Poincar\'e duality and excision $$H^{k(d-1)+1}(W,E_{d}(\ell);\Z_{2}) \simeq H_{(n-k)(d-1)-1} (W-E_{d}(\ell);\Z_{2})=0$$ 
(since by Proposition \ref{prop44} nontrivial homology of $W-E_d(\ell)$ is concentrated in degrees divisible by $d-1$)
and hence the map 
$$i^{k}: H^{k(d-1)}(W;\Z_{2}) \rightarrow H^{k(d-1)}(E_{d}(\ell);\Z_{2})$$ 
induced by inclusion is surjective. On the other hand, from the exact sequence  
$$H^{k(d-1)}(W,E_{d}(\ell);\Z_{2}) \to H^{k(d-1)}(W;\Z_{2}) \stackrel{i^{k}}\rightarrow H^{k(d-1)}(E_{d}(\ell);\Z_{2})$$ 
and the commutative square

%\begin{xy} \xymatrix{ \hspace{50pt}
%H^{k(d-1)}(W,E_{d}(\ell); \Z_{2}) \ar[d]^/0em/{PD}  
%\ar[r] & H^{k(d-1)}(W;\Z_{2}) \ar[d]^/0em/{PD}  
%  \\ \hspace{50pt}
%H_{(n-k)(d-1)}(W-E_{d}(\ell);\Z_{2}) \stackrel{j_{n-k}}\rightarrow %\ar[r] 
%& H_{(n-k)(d-1)}(W;\Z_{2}) 
%}\end{xy} 
%$$
%\begin{array}{ccc}
%H^{k(d-1)}(W,E_{d}(\ell); \Z_{2}) & \to & H^{k(d-1)}(W;\Z_{2})\\  \\
%\downarrow {PD}  && \downarrow {PD}\\ \\
%H_{(n-k)(d-1)}(W-E_{d}(\ell);\Z_{2}) & \stackrel{j_{n-k}}\rightarrow  & H_{(n-k)(d-1)}(W;\Z_{2}) 
%\end{array}
%$$

\begin{xy} \xymatrix @C=0.5in{ \hspace{50pt}
H^{k(d-1)}(W,E_{d}(l); \Z_{2}) \ar[d]^/0em/{PD}  
\ar[r] & H^{k(d-1)}(W;\Z_{2}) \ar[d]^/0em/{PD}  
  \\ \hspace{50pt}
H_{(n-k)(d-1)}(W-E_{d}(l);\Z_{2}) 
\ar[r]^/1.8em/{j_{n-k}} & H_{(n-k)(d-1)}(W;\Z_{2}) 
}\end{xy}

where the columns are Poincar\'e duality maps, we see that the kernel of $i^{k}$ consists exactly of the Poincar\'e duals of the elements of the image of $j_{n-k}$.
  
We identify $H^{*}(W; \Z_{2})$ with the exterior algebra generated by the classes $$X_{j}=\pi^{*}_{j}[S^{d-1}]\in H^{d-1}(W;\Z_2),$$ 
where $\pi_{j}: W=(S^{d-1})^{n} \rightarrow S^{d-1}$ is the projection onto the $j$-th factor, $j=1, \dots, n$ and $[S^{d-1}] \in H^{d-1}(S^{d-1}; \Z_{2})$ is the fundamental class. Then 
$$H^{(d-1)*}(E_{d}(\ell);\Z_{2}) \simeq \Lambda_d (X_{1}, \dots, X_{n})/I$$ where $I=\text{ker }i^{k}$.

By Corollary \ref{Image} the kernel of $i^{k}$ consists of linear combinations
of the Poincar\'e duals of all the classes $[V_{J}],[W_{J}] \in H_{(d-1)*}(W;\Z_{2})$ where $J$   runs over long subsets of $\{1, \dots, n\}$ containing the index $n$. 

It is obvious that the Poincar\'e dual of the class $[V_{J}]$ is the monomial $$X^{J}=X_{j_{1}} \dots X_{j_{k}}, \quad \text{where}\quad  J=\{j_{1}, \dots, j_{k}\}.$$

From the definition of $W_J$ it follows that%\marginpar{explanation needed}
\begin{eqnarray}\label{www}[W_{J}]=\sum \limits_{j \in J}\,  [V_{J_j}]\in H_\ast(W;\Z_2),\end{eqnarray}
where for $j\in J$ the symbol $J_j$ denotes $J-\{j\}$. 
This follows from the well known fact that the $\Z_2$-homology class of the diagonal in a product of spheres equals the sum of the fundamental classes of spheres in the product.
To apply this observation we note that $W_J$ can be viewed as a product of the diagonal $\Delta \subset \prod_{i\in J}S_i^{d-1}$ times $\prod_{i\notin J}S_i^{d-1}$ where
$S^{d-1}_i$ denotes the sphere $\{u_j=p; \mbox{for all} \, j\not=i\}\subset W$. 

It follows from (\ref{www}) that the Poincar\'e dual of the class $[W_{J}]$ is the cohomology class 
$$\sum \limits_{j \in J} X^{J_j}\in H^\ast(W;\Z_2).$$

Hence we obtain that the kernel of $i^{k}$ is the ideal $I$ which is additively generated by all the monomials $X^{J}=X_{j_{1}} \dots X_{j_{k}}$ 
so that $J$ is long with respect to $\ell$ and $n\in J$ 
together with the polynomials of the form 
$\sum \limits_{j \in J} X^{J_{j}}$ where the subset $J\subset \{1, \dots, n\}$ is long with respect to $\ell$ and $n \in J$.

%so that $J=\{j_{1}, \dots, j_{k}\} \subset \{1, \dots, n\}$ is long and $n \in J$ 

Now we define a new multiplicative basis $Z_{1}, \dots, Z_{n}$ of $H^{*}(W;\Z_{2})$ as follows
$$Z_{j}=X_{j}+X_{n} \quad \mbox{for}\quad j=1, \dots, n-1, \quad \mbox{and}\quad Z_{n}=X_{n}.$$ 
Then for a subset $J=\{j_1, \dots, j_k\}$ 
of $\{1, \dots, n\}$ %with $j_1<j_2< \dots<j_k$ 
the monomial 
$Z^J=Z_{j_1}\dots Z_{j_k}$ equals
$$Z^J= \left\{
\begin{array}{lll}
X^J, & \mbox{if} & n\in J,\\
X^J +\left(\sum_{j\in J}X^{J_j}\right)X_n, &\mbox{if} &n\notin J.
\end{array}\right.
$$
The last equality may also be expressed as follows. For a subset $J\subset \{1, \dots, n\}$ containing $n$ denote $J'=J-\{n\}$. Then $$Z^{J'}=\sum_{j\in J}X^{J_j}.$$

We see that in the new basis $Z_1, \dots, Z_n$ the ideal $I\subset \Lambda_d(Z_1, \dots, Z_n)$ is additively generated by the monomials 
$Z^J$ where $J\subset \{1, \dots, n\}$ either contains $n$ and is long with respect to $\ell$ or $J$ does not contain $n$ and $J\cup \{n\}$ is long with respect to $\ell$. 
Hence, the set of monomoals $Z^J$ where $n\notin J$ and $J\cup \{n\}$ is long with respect to $\ell$ forms a multiplicative basis of $I$. 

This completes the proof of Proposition \ref{ECohomology}. \qed

\section{Proof of Theorem \ref{Wal}}

The proof of Theorem \ref{Wal} relies on Proposition \ref{ECohomology} and on 
an algebraic result of J. Gubeladze \cite{Gub} which we recall below. 
Given a commutative ring $R$, an ideal $I \subset R[Z_{1}, \dots, Z_{n}]$ of the polynomial ring is {\it a monomial ideal} if it is generated by a set of monomials 
$X^{a_{1}}_{1} \dots X^{a_{m}}_{m}$, where $a_{i} \geq 0$.

\begin{theorem}[\cite{Gub}]
Let $R$ be a commutative ring and let $I \subset R[X_{1}, \dots, X_{m}]$, $I' \subset R[Y_{1}, \dots, Y_{m'}]$ be two monomial ideals. 
Assume that $I \cap \{X_{1}, \dots, X_{m}\}=\emptyset$, $I' \cap \{Y_{1}, \dots, Y_{m'}\}=\emptyset$ and there is an isomorphism $$R[X_{1}, \dots, X_{m}]/I \simeq  R[Y_{1}, \dots, Y_{m'}]/I'$$ of $R$-algebras. 
Then $m=m'$ and there is a bijection $$\{X_{1}, \dots, X_{m}\} \rightarrow \{Y_{1}, \dots, Y_{m'}\}$$ which maps $I$ to $I'$. 
\end{theorem} 

{\it Proof of Theorem \ref{Wal}}. The implications (a) $\Rightarrow$ (b) $\Rightarrow$ (c) are obvious and the implication (d) $\Rightarrow$ (a) was proven earlier, see Proposition \ref{prop32}. Hence the main issue is to prove the implication (c) $\Rightarrow$ (d). 

Let $\ell, \ell'$ be two generic length vectors so that the graded cohomology rings $H^{(d-1)\ast}(E_{d}(\ell);\Z_2)$ and $H^{(d-1)\ast}(E_{d}(\ell');\Z_2)$ are isomorphic. 
Without loss of generality we can assume that both $\ell$ and $\ell'$ are ordered.

By Proposition \ref{ECohomology}, the algebra $H^{(d-1)*}(E_{d}(\ell);\Z_{2})$ has the form 
$$\Z_{2}[Z_{1}, \dots, Z_{n}]/K$$ 
where $K$ is the monomial ideal generated multiplicatively by the squares $Z^{2}_{1}, \dots, Z^2_n$ 
as well as by monomials $Z_{j_{1}} \dots Z_{j_{k}}$ such that $j_1<\dots < j_k<n$ and the set 
$\{j_1, j_2, \dots, j_k, n\}$ 
is long with respect to $\ell$. 

It is well-known that $E_{d}(\ell) = \emptyset$ iff the set $\{n\}$ is long with respect to $\ell$. Clearly any two ordered length vectors $\ell$ with this property lie in the same chamber 
and hence we may assume below that $E_{d}(\ell) \not= \emptyset \not=E_{d}(\ell')$.
Hence $Z_{n} \notin K$. Furthermore for $j =1, \dots, n-1$ we have $Z_{j} \in K$ iff the two-element subset $\{j,n\}$ is long with respect to $\ell$ . Denote 
$$i=\max \{j: \text{ }\{j,n\} \text{ is short w. r. t. }\ell\},$$ 
$$i'=\max \{j: \text{ }\{j,n\} \text{ is short w. r. t. }\ell'\}.$$
There is an isomorphism 
$$\Z_{2}[Z_{1}, \dots, Z_{n}]/K \simeq \Z_{2}[Z_{1}, \dots, Z_{i},Z_{n}]/I,$$ 
where $I \subset \Z_{2}[Z_{1}, \dots, Z_{i},Z_{n}]$ is the ideal generated multiplicatively by the squares $Z^{2}_{j}$ for $j \in \{ 1, \dots, i,n \}$ and by 
the monomials $Z_{j_{1}} \dots Z_{j_{k}}$ so that $J=\{j_{1}, \dots, j_{k}\}\subset \{ 1, \dots, i\}$ and $J\cup \{n\}$ is long w. r. t. $\ell$.

Similarly, the ring $H^{(d-1)*}(E_{d}(\ell');\Z_{2})$ is isomorphic to 
$\Z_{2}[Z_{1}, \dots, Z_{i'},Z_{n}]/I',$ where the monomial ideal $I'$ is defined analogously 
to $I$ with the words {\it \lq\lq long w.r.t. $\ell$\rq\rq} replaced by {\it \lq\lq long w.r.t. $\ell'$\rq\rq.}

Using Gubeladze's Theorem, an isomorphism
$$\Z_{2}[Z_{1}, \dots, Z_{i},Z_{n}]/I \simeq \Z_{2}[Z_{1}, \dots, Z_{i'},Z_{n}]/I'$$
implies $i=i'$ and the existence of a permutation $\sigma$ of $\{1, \dots, i, n\}$ so that a subset $J \subset \{1, \dots, i, n\}$ with $n \in J$ is long with respect to $\ell$ if and only if the subset $\sigma(J)$ is long with respect to $\ell'$. 

Extend $\sigma$ to a permutation of $\{1, \dots, n\}$ by defining $\sigma(j)=j$ for $j=i+1, \dots, n-1$. Let $\alpha$ denote the transposition of the indices $\sigma(n)$ and $n$ and let $\alpha(\ell')$ denote the length vector obtained from $\ell'$ by interchanging the entries with these two indices.

We know that a subset $J \subset \{1, \dots, n\}$ with $n \in J$ and $i+1, \dots, n-1 \notin J$ is long with respect to $\ell$ if and only if $(\alpha \circ \sigma)(J)$ is long with respect to 
$\alpha(\ell')$. Since for every $j \in \{ i+1, \dots, n-1 \}$, the subset $\{j,n\}$ is long with respect to both $\ell$ and $\alpha(\ell')$, we see that under $\alpha \circ \sigma$ the sets 
$$\{J \subset \{1, \dots, n\}: J \, \text{is long with respect to}\,  \ell \, \, \mbox{and}\,  \, n \in J\}$$ 
and 
$$\{J \subset \{1, \dots, n\}: J \, \text{is long with respect to}\,  \alpha(\ell') \, \, \mbox{and}\,  \, n \in J\}$$ 
%$$\{J \subset \{1, \dots, n\}: J \, \text{is long w. r. t. }\ell', n \in J\}$$ 
%$$\{J \subset \{1, \dots, n\}: \text{ }n \in J, \, J \text{ long w. r. t. } \alpha(\ell')\}$$ 
are identified. 

Next we recall that according to Lemma 4 from \cite{Fa4} two generic length vectors $\ell, \ell' \in \R^n_>$  
lie in the same chamber if and only if the following condition holds: a subset $J \subset \{1, \dots, n\}$ containing
$n$ is short with respect to $\ell$ if and only if it is short with respect to $\ell'$. 

Returning to the proof and applying Lemma 4 in \cite{Fa4} we see that $\ell$ and $\alpha(\ell')$ lie in the same chamber. Hence $\ell$ and $\ell'$ lie in the same chamber after a permutation of the entries. 

The proof of Theorem \ref{Wal} is now complete.   \qed     

We believe that a generalization of Theorem \ref{Wal} allowing nongeneric length vectors is also true: 
if $d>2$ and 
two length vectors $\ell$ and $\ell'$ are such that 
the graded algebras $H^\ast(E_d(\ell);\Z_2) $ and $H^\ast(E_d(\ell');\Z_2) $ are isomorphic then for some permutation 
$\sigma$ the vectors $\ell$ and $\sigma(\ell')$ {\it lie in the same stratum}, i.e. have identical sets of short, long and median subsets. 
This can be proven by using Proposition \ref{ECohomology}, the Theorem of Gubelabze and the Poincar\'e duality defect (as in Theorem 7 of \cite{FHS}). 

\section{Betti numbers of spaces of polygons}
The above arguments also allow to compute explicitly the $\Z_{2}$-Betti numbers of the spaces $E_{d}(\ell)$ in terms of the length vector $\ell$. 
These formulae use certain combinatorial quantities from \cite{Fa2} which 
 we now recall. 
For a length vector $\ell=(l_1, \dots, l_n)\in \R^n_>$, where $l_1\le \dots \le l_n$, 
the quantity $a_{k}(\ell)$ denotes the number of subsets $J \subset \{1, \dots, n\}$ which are short with respect to $\ell$ with $n \in J$ and $|J|=k+1$.
Similarly, $b_k(\ell)$ denotes the number of subsets $J \subset \{1, \dots, n\}$ which are median with respect to $\ell$ with $n \in J$ and $|J|=k+1$.

\begin{proposition}\label{EBettiNumbers}
If $d>2$ then for any ordered length vector $\ell\in \R^n_>$  one has
$$\dim_{\Z_2} H_{(d-1)k}(E_{d}(\ell);\Z_{2})=a_{k}(\ell)+b_{k}(\ell)+a_{k-1}(\ell)+b_{k-1}(\ell)$$
for $k=0, \dots, n-2$. Furthermore,
$$\dim_{\Z_2} H_{(d-1)k-1}(E_{d}(\ell);\Z_{2})=a_{n-k-2}(\ell)+a_{n-k-1}(\ell)$$ for $k=1, \dots n-1$.
All other homology groups of the space $E_d(\ell)$ vanish. 
\end{proposition}
\begin{proof}
As was shown in the proof of Proposition \ref{ECohomology}, the dimensions of the kernel and of the cokernel of the homomorphism 
$$j_{k} \colon H_{(d-1)k}(W-E_{d}(l);\Z_{2}) \rightarrow H_{(d-1)k}(W;\Z_{2})$$
induced by the inclusion are given by
$$\dim \ker  j_{k}=\dim A'_{k}+\dim B'_{k} =a_{k-2}(\ell)+a_{k-1}(\ell)$$ 
and $$\dim {\rm coker}  j_{k}=\dim C_{k}+\dim D_{k}=a_{n-k}(\ell)+b_{n-k}(\ell)+a_{n-k-1}(\ell)+b_{n-k-1}(\ell).$$ 
Using the homological long exact sequence of the pair $(W, W - E_{d}(\ell))$, excision and Poincar\'e duality (applied to a complement of an open regular neighbourhood of
$E_d(\ell)$ in $W$), one concludes that 

\begin{eqnarray*}
\dim H_{(d-1)k}(E_{d}(\ell);\Z_{2})&=&\dim H^{(d-1)k}(E_{d}(\ell);\Z_{2})\\ &=&\dim H_{(d-1)(n-k)}(W,W - E_{d}(\ell);\Z_{2})\\ & =& \dim {\rm coker} j_{n-k}
\\ &=& a_{k}(\ell)+b_{k}(\ell)+a_{k-1}(\ell)+b_{k-1}(\ell)
\end{eqnarray*}
and 
\begin{eqnarray*}
\dim H_{(d-1)k-1}(E_{d}(\ell);\Z_{2})&=&\dim H^{(d-1)k-1}(E_{d}(\ell);\Z_{2})\\
&=&\dim H_{(d-1)(n-k)+1}(W,W - E_{d}(\ell);\Z_{2})\\
&=&\dim \ker j_{n-k}\\ &=& a_{n-k-2}(\ell)+a_{n-k-1}(\ell).
\end{eqnarray*}
The vanishing of all other homology groups of $E_d(\ell)$ 
follows from similar arguments using the observation that homology of $W$ and of $W-E_d(\ell)$ are trivial in all dimensions non-divisible by $d-1$. 
This completes the proof. 
\end{proof}

\begin{remark}\label{remone}
By Proposition \ref{EBettiNumbers} for $d>2$ the first two possibly non-vanishing Betti numbers of $E_{d}(\ell)$ 
are in dimensions $0$ and $d-2$. From the given formulae it is immediate that the zero-dimensional Betti number is always one assuming that $E_d(\ell)\not=\emptyset$. 
We obtain that for $d>2$ the space $E_{d}(\ell)$ is connected if it is nonempty. 
\end{remark}

\begin{remark}\label{remtwo}
For $d>2$ the $(d-2)$-Betti number can be expressed as
$$ \dim H_{d-2}(E_{d}(\ell);\Z_{2})=a_{n-3}(\ell)+a_{n-2}(\ell).$$
Recall that $a_{n-2}(\ell)$ is the number of short subsets $J \subset \{1, \dots, n\}$ with $n \in J$ and $|J|=n-1$. 
Equivalently, $a_{n-2}(\ell)$ is the number of one-element long subsets of $\{1, \dots, n-1\}$. 
Thus, if $E_{d}(\ell) \neq \emptyset$, we have $a_{n-2}(\ell)=0$. 
Similarly, $a_{n-3}(\ell)$ can be identified with the number of two-element  long subsets $J \subset \{1, \dots, n-1\}$. 
There can be at most one subset $J$ of this type, namely if $\ell$ is ordered then we must have $J=\{n-2,n-1\}$. 
Thus the homology group $H_{d-2}(E_{d}(\ell);\Z_{2})$ either vanishes or has dimension one.

Moreover, in the case when $d>2$ and $\dim H_{d-2}(E_{d}(\ell);\Z_{2})\not=0$ one may characterize the diffeomorphism type of $E_d(\ell)$ completely. Namely, in this case $E_d(\ell)$ is diffeomorphic to the product 
$$V_2(\R^d) \times S^{d-1} \times \dots \times S^{d-1}$$
of the Stiefel manifold $V_2(\R^d)$ and $n-3$ copies of the sphere $S^{d-1}$. We will leave this statement as an exercise for the reader. 
\end{remark}

Let us now consider some examples:

{\it Example 1.} For $n\ge 3$ consider the following length vector
$$l_{1}= \dots = l_{n-1}= a, \quad l_{n}=1-a(n-1)$$ 
with 
$$\frac{1}{2(n-1)}\, \leq\,  a \, \leq \, \min\left\{\frac{1}{2(n-2)}, \frac{1}{n}\right\}.$$
Here every subset $J \subset \{1, \dots, n\}$ with $n \in J$ and $|J|>1$ is long. Hence by Proposition \ref{EBettiNumbers} the nonvanishing Betti numbers are in dimensions $0$,  $d-1$, $(n-2)(d-1)-1$ and $(n-1)(d-1)-1$ and all are equal to $1$. 
Let us determine the space of polygons $E_d(\ell)$ explicitly in this case. 

%Let us show that in this case $E_{d}(\ell)$ is in fact an $S^{(n-2)(d-1)-1}$ bundle over $S^{d-1}$.

Let $\ell'$ denote the length vector obtained from $\ell$ by deleting the last component and  consider the function 
$$f_{\ell'}: (S^{d-1})^{n-1} \rightarrow \R,$$
$$(u_{1}, \dots, u_{n-1}) \mapsto -\left|\sum \limits_{j=1, \dots, n-1} l_{j} u_{j}\right|^{2}.$$ 
We have $$E_{d}(\ell)=f_{\ell'}^{-1}(-l^{2}_{n}).$$ 
Furthermore $f_{\ell'}$ has minimum value $-L^{2}$ where 
$L=a(n-1)$ 
and the preimage $f_{\ell'}^{-1}(-L^{2})$ is the diagonal 
$\Delta \subset (S^{d-1})^{n-1}$. 
We note that $\Delta$ is a submanifold of codimension $(n-2)(d-1)$. 
Under our assumptions on the value of $a$ every $n$-dimensional length vector of the form $(a, \dots, a, l'_{n})$ with $l_{n} \leq l'_{n}< L$  is generic. Hence%\marginpar{This "Hence" should be explained} 
$E_{d}(\ell)$ is the unit sphere bundle of the normal bundle of the diagonal $\Delta \subset (S^{d-1})^{n-1}$. It is easy to see that this normal  bundle $\eta$ is isomorphic to the Whitney sum of $n-2$ copies of the tangent bundle $\tau$ of $S^{d-1}$. 

In the case when $n=3$ the unit sphere bundle of $\tau$ can be identified with the Stiefel manifold $V_2(\R^d)$. Thus we obtain that for $n=3$ the manifold of polygons 
$E_d(\ell)$ is diffeomorphic to $V_2(\R^d)$. This is also consistent with our Remark \ref{remtwo} above since the two-element subset $\{n-2, n-1\}$ is long for $n=3$. 

In the case $n\ge 4$ the bundle $\eta= \tau\oplus \tau \oplus \dots\oplus \tau$ 
(having $n-2$ terms) is trivial. Indeed, denoting by $\theta^k$ the trivial bundle of rank $k$ over $S^{d-1}$ we 
observe that $\tau\oplus \theta^1$ is trivial and hence $\eta\oplus \theta^{n-2}$ is trivial. Since for $n\ge 4$ the rank $(d-1)(n-2)$ of $\eta$ is greater than the dimension $d-1$ of the base,
we obtain that for $n\ge 4$ triviality $\eta\oplus\theta^{n-2}\simeq \theta^{d(n-2)}$ implies that the bundle $\eta\simeq \theta^{(d-1)(n-2)}$ is trivial (see \cite{Hu}, Chapter 9, Theorem 1.5). Hence 
$$E_d(\ell) \simeq S^{(d-1)(n-2)-1}\times S^{d-1}$$
under the conditions indicated above.

{\it Example 2.} The two length vectors
$$\ell=(1,2,2,2,4,4)$$
and 
$$\ell'=(1,1,3,4,8,8)$$ 
are generic and the quantities $a_{k}(\ell)$ and $a_{k}(\ell')$ coincide for all $k$. Hence by the formulae of Proposition \ref{EBettiNumbers} the $\Z_{2}$-Betti numbers of 
$E_{d}(\ell)$ and of $E_{d}(\ell')$ are the same. However, $\ell$ and $\ell'$ lie in different chambers
since the set $J=\{1, 4,6\}$ is short with respect to $\ell$ but is long with respect to $\ell'$. 

If two ordered length vectors $\ell$ and $\ell'$ are such that for some permutation 
$\sigma: \{1, \dots, n\}\to \{1, \dots, n\}$ the vectors $\ell$ and $\sigma(\ell)$ lie in the same chamber, then the vectors $\ell$ and $\ell'$ lie in the same chamber. This is a consequence 
of Lemma 3 from \cite{FHS} which formally requires that the permutation $\sigma$ fixes one of the indices but the proof applies to the case of an arbitrary permutation. 

This example shows that in general the $\Z_{2}$-Betti numbers of the space $E_{d}(\ell)$ do not determine the orbit of the chamber of $\ell$ under the $\Sigma_n$-action.

\vspace{2ex}

 \begin{center}

\end{center}
\vspace{2ex}


\begin{thebibliography}{9999999}
%\bibitem{Fa1} M. Farber, Topology of Closed One-Forms, Mathematical Surveys and Monographs Volume 108, American Mathematical Society, 2004
\bibitem{Fa2} M. Farber, D. Schuetz, {\it Homology of Planar Polygon Spaces}, Geom. Dedicata 125 (2007), 75-92.

\bibitem{Fa4} M. Farber, J.-Cl. Hausmann, D. Schuetz, {\it On the Conjecture of Kevin Walker}, J. of Topology and Anlysis 1 (2009), 65-86.

\bibitem{FHS} M. Farber, J.-Cl. Hausmann and D. Schuetz, \textit{The Walker conjecture for chains in $\R^d$}, to appear in \lq\lq Math. Proc. Camb. Phil. Soc.\rq\rq. 

\bibitem{Fa5} M. Farber, {\it Invitation to Topological Robotics}, Zurich Lectures in Advanced Mathematics, EMS, 2008.

\bibitem{Gub} J. Gubeladze, {\it The Isomorphism Problem for Commutative Monoid Rings}, Journal of Pure and Applied Algebra 129 (1998), 35-65.

\bibitem{Hau} J.-C. Hausmann, {\it Sur la topologie des bras articul\'es}, in Algebraic Topology, Springer Lecture Notes, 1474 (1989), 146 - 159.

\bibitem{HK}
J.-C. Hausmann and A. Knutson, \textit{The cohomology rings of polygon
spaces}, {\em Ann. Inst. Fourier (Grenoble)} \textbf{48} (1998),
281--321.

\bibitem{Hu} D. Husemoller, {\it Fibre Bundles,} 3rd ed., Graduate Texts in Mathematics,
vol. 20, Springer-Verlag, New York, 1994.

\bibitem{JS} D. Jordan and M. Steiner, \textit{Configuration spaces of
mechanical linkages}, Discrete and Computational Geometry, {\bf 22}(1999), 297 - 315.

\bibitem{Ka}
 Y. Kamiyama, \textit{The homology of singular polygon spaces}, {\em
 Canad. J. Math.} \textbf{50} (1998), 581--594.

\bibitem{KM1} M. Kapovich, J.L. Millson, \textit{On the moduli space of
polygons in the Euclidean plane}, J. Diff. Geometry \textbf{42}(1995), 133-164.

\bibitem{KM3}  M. Kapovich, J.L. Millson, \textit{Universality theorems for
configuration spaces of planar linkages}, Topology \textbf{41}(2002),
1051-1107.

\bibitem{Sch} D. Schuetz, {\it The Isomorphism Problem for Planar Polygon Spaces}, Journal of Topology (2010) 3(3), 713-742.

\bibitem{Wa} K. Walker, {\it Configuration Spaces of Linkages}, Bachelor's Thesis, Princeton (1985).
\end{thebibliography}
\end{document}